\documentclass[12pt, leqno, british]{amsart}
\usepackage[
	style=alphabetic,
	citestyle=alphabetic,
	backend=biber,
    maxbibnames=4,
    doi=false,isbn=false,url=false
]{biblatex}
\usepackage[foot]{amsaddr}
\renewbibmacro{in:}{%
  \ifentrytype{article}{}{\printtext{\bibstring{in}\intitlepunct}}}
\usepackage{a4, amsmath}
\usepackage{mathtools}
\usepackage{amssymb, amsfonts,amsthm, amscd, mathdots}
\usepackage{enumerate}
\usepackage{hyperref, cleveref}
\usepackage{csquotes}
\usepackage{color}
\usepackage{datetime}
\usepackage{tikz,graphicx,multirow,multicol,paralist,enumitem,bm,relsize}

\usepackage{geometry}
\theoremstyle{plain}
\newtheorem{thm}{Theorem}[section]
\newtheorem{thmI}{Theorem}

\newtheorem{lem}[thm]{Lemma}
\newtheorem{prp}[thm]{Proposition}
\theoremstyle{definition}
\newtheorem{dfn}[thm]{Definition}
\newtheorem{ques}[thm]{Question}
\theoremstyle{remark}
\newtheorem{rmk}[thm]{Remark}
\newtheorem{eg}[thm]{Example}
\numberwithin{equation}{section}
\crefname{thm}{theorem}{theorems}
\crefname{cor}{corollary}{corollaries}
\crefname{lem}{lemma}{lemmata}
\crefname{prp}{proposition}{propositions}

\newcommand{\C}{\mathbb{C}}
\newcommand{\cc}{\mathbb{C}}
\newcommand{\Q}{\mathbb{Q}}
\newcommand{\qq}{\mathbb{Q}}
\newcommand{\R}{\mathbb{R}}
\newcommand{\rr}{\mathbb{R}}
\newcommand{\Z}{\mathbb{Z}}
\newcommand{\zz}{\mathbb{Z}}

\newcommand{\co}{\mathcal O}

\newcommand{\mc}{\mathcal}
\newcommand{\mf}{\mathfrak}

\newcommand{\cbm}[2]{\left\{{#1}\; \middle|\; {#2}\right\}}

\newcommand{\rb}[1]{\left({#1}\right)}

\newcommand{\tr}{\operatorname{tr}}

\def\house#1{{%
    \setbox0=\hbox{$#1$}
    \vrule height \dimexpr\ht0+1.4pt width .4pt depth \dp0\relax
    \vrule height \dimexpr\ht0+1.4pt width \dimexpr\wd0+2pt depth \dimexpr-\ht0-1pt\relax
    \llap{$#1$\kern1pt}
    \vrule height \dimexpr\ht0+1.4pt width .4pt depth \dp0\relax
}}

\newcommand{\ba}{\[\begin{aligned}} 
\newcommand{\ea}{\end{aligned}\]} 
\newcommand{\bd}{\begin{tikzcd}} 
\newcommand{\ed}{\end{tikzcd}} 
\newcommand{\bdd}{\begin{center}\begin{tikzcd}} 
\newcommand{\edd}{\end{tikzcd}\end{center}} 
\newcommand{\bdp}{\begin{center}\begin{tikzpicture}} 
\newcommand{\edp}{\end{tikzpicture}\end{center}} 
\newcommand{\bi}{\begin{itemize}} 
\newcommand{\ei}{\end{itemize}} 
\newcommand{\bea}{\begin{enumerate}[label=(\alph*)]} 
\newcommand{\ben}{\begin{enumerate}[label=(\arabic*)]} 
\newcommand{\ber}{\begin{enumerate}[label=(\roman*)]} 
\newcommand{\beani}{\begin{enumerate}[label=(\alph*), wide, labelwidth=!, labelindent=0pt]} 
\newcommand{\benni}{\begin{enumerate}[label=(\arabic*), wide, labelwidth=!, labelindent=0pt]} 
\newcommand{\berni}{\begin{enumerate}[label=(\roman*), wide, labelwidth=!, labelindent=0pt]} 
\newcommand{\ee}{\end{enumerate}} 

\title[Most totally real fields do not have  universal forms or Northcott property]{Most totally real fields do not have \\ universal forms or Northcott property}
\author[N.~Daans]{Nicolas Daans}
\author[V.~Kala]{Vítězslav Kala}
\author[S.~H.~Man]{Siu Hang Man}
\author[M.~Widmer]{Martin Widmer}
\author[P.~Yatsyna]{Pavlo Yatsyna}

\address[N.D., V.K., S.H.M., P.Y.]{Charles University, Faculty of Mathematics and Physics, Department of Algebra, Sokolov\-ská 83, 186~75 Praha~8, Czech Republic}
\address[N.D.]{KU Leuven, Faculty of Science, Department of Mathematics, Celestijnenlaan 200B, 3001 Heverlee, Belgium}
\address[M.W.]{Graz University of Technology, Institute of Analysis and Number Theory, Steyrergasse 30/II, 8010 Graz, Austria}

\email[N.~Daans]{nicolas.daans@kuleuven.be}
\email[V.~Kala]{vitezslav.kala@matfyz.cuni.cz}
\email[S.~H.~Man]{shman@karlin.mff.cuni.cz}
\email[M.~Widmer]{martin.widmer@tugraz.at}
\email[P.~Yatsyna]{p.yatsyna@matfyz.cuni.cz}

\thanks{The authors were supported by Czech Science Foundation GAČR, grant 21-00420M (N.D., V.K., S.H.M.), Charles University programmes PRIMUS/24/SCI/010 (N.D., S.H.M, P.Y.), PRIMUS/25/SCI/008 (S.H.M.), and UNCE/24/SCI/022 (P.Y.). \newline
The authors declare no competing interest.
}

\begin{document}
\begin{abstract}
We show that, in the space of all totally real fields equipped with the constructible topology, the set of fields that admit a universal quadratic form, or have the Northcott property, is meager.
The main tool is a new theorem on the number of square classes of totally positive units represented by a quadratic lattice of a given rank. \vspace{\baselineskip}\newline
\noindent\textsc{Significance statement.}
The classical fact that every positive integer is a sum of four squares led to the much more general study of universal quadratic forms. Their behavior over extensions of the integers by elements such as the square root of $2$ crucially depends on the existence of small elements in the extension. This is captured well by the Northcott property, which has been influential since the 1960s due to its connection to decidability of solving diophantine equations. In this paper we consider extensions of infinite degree, and establish that fields which admit a universal form or have the Northcott property are very sparse among them. Along the way, we obtain a new criterion connecting units in the extension with universal forms.
\end{abstract}

\makeatletter
\@namedef{subjclassname@2020}{%
  \textup{2020} Mathematics Subject Classification}
\makeatother

\maketitle

\section{Introduction}

Since the 18th century, Lagrange's four square theorem on the universality of $X^2 + Y^2 + Z^2+ W^2$ led mathematicians to wonder what other quadratic forms represent all positive integers, i.e. are \textit{universal}. After the classification works by Ramanujan \cite{Ramanujan}, Dickson \cite{Di1}, Willerding \cite{Wi}, Conway--Schneeberger \cite{Bhargava}, among many others, Bhargava--Hanke \cite{Bhargava-Hanke} wrapped up the story over $\Z$ by proving the 290--Theorem, stating that a positive definite quadratic form is universal if it represents $1,2,3,\dots,290$.

The analogue of Lagrange's four square theorem does not hold when we replace the rational integers $\Z$ by the algebraic integers $\co_K$ of a totally real number field $K$. To be more precise, let $\co_K^+$ be the set of totally positive algebraic integers in $K$ (i.e.~those that are mapped to a positive element by all the embeddings $K\hookrightarrow \R$). 
Given a \textit{quadratic form $Q\in\co_K[X_1,\dots,X_n]$ over $K$} and $\alpha \in \co_K$,
we say that $Q$ \textit{represents $\alpha$} if $Q(x) = \alpha$ for some $x \in \co_K^n$. Further, $Q$ is \textit{positive definite} if $Q(x)$ is totally positive for every $0\neq x\in\co_K^n$, and \textit{universal (over $K$)} if, moreover, it represents each $\alpha\in\co_K^+$.

In the 1940s, Siegel \cite{Siegel} proved that in every totally real number field $K$ different from $\Q$ and $\Q(\sqrt 5)$, there exists a totally positive algebraic integer which is not the sum of any number of squares of algebraic integers in $K$, and Maa{\ss} \cite{Maass} showed that in $K = \Q(\sqrt 5)$, the sum of three squares (i.e.~the quadratic form $X^2 + Y^2 + Z^2$) is universal.
Nevertheless, over every totally real number field, there exists \textit{some} universal quadratic form thanks to the asymptotic local-global principle \cite{HKK} (see, e.g.~\cite[Section 5]{KalaSurvey} for a proof sketch), but these universal forms often require many variables, as shown by numerous recent works, e.g.~\cite{BK,KT,Yat19}. In particular, for every $n$, the set of squarefree integers $D>1$ such that the
real quadratic field $\Q(\sqrt D)$ admits an $n$-ary universal form has natural density 0 \cite{Kala-Yatsyna-Zmija}; 
a similar result also holds for multiquadratic fields \cite{Man}. 
Despite this progress, Kitaoka's influential conjecture that there are only finitely many totally real number fields with a ternary universal quadratic form remains open (see, for example, \cite{KalaYatsyna2}).
For contrast, let us note that the situation is very different over number fields which are not totally real; see \Cref{sect:non-tot-real} for details.

In fact, the above definition of a universal quadratic form over $K$ works perfectly well for any field $K$ of totally real algebraic numbers, not necessarily of finite degree over $\Q$. However, 
the existence of universal quadratic forms becomes more precarious when one considers rings of integers in totally real extensions $K/\Q$ of \textit{infinite} degree.
Recently, the first three authors \cite{DKM-Northcott} showed that, depending on $K$, universal quadratic forms over $K$ may or may not exist. For example, the unary form $X^2$ is universal over the field of all totally real algebraic numbers $\Q^{\tr}$, whereas there is no universal form over $\Q^{\tr,2}=\Q(\sqrt n\mid n\in\Z_{>0})$.
However, the results of \cite{DKM-Northcott} left completely open the question of arithmetic statistics, i.e. \textit{how often do infinite extensions admit a universal form?}
We now resolve it by showing that the answer is \textit{rarely!} 
As there are uncountably many totally real infinite extensions, we cannot just order them by an invariant such as the discriminant, but we need to proceed topologically.

To make this precise, consider the set $\mc X$ of all subfields of $\Q^{\tr}$.
Since every totally real field can be embedded into $\Q^{\tr}$, we may think of $\mc X$ as the set of all totally real fields.
We endow $\mc X$ with the \emph{constructible topology}, i.e.~the coarsest topology for which the sets
$ \lbrace K \in \mc X \mid a \in K \rbrace $
are both open and closed, for all $a \in \Q^{\tr}$.
With respect to this topology, we obtain the following result.

\begin{thmI}[see \Cref{prp:uqf-meager}]\label{TI:meager}
The set 
of totally real fields which admit a universal quadratic form 
is a meager subset of $\mc X$.
\end{thmI}

Recall that a subset of a topological space is called \emph{meager} if it is a countable union of nowhere dense subsets (see \Cref{sect:topological})
and \emph{comeager} if its complement is meager.
In a Baire space non-empty open sets are non-meager and comeager sets are dense. 
Therefore, meagerness is a natural notion of ``smallness'' in a Baire space, and
it turns out that $\mc X$ is a Baire space. 
Our approach is inspired by \cite{DF21} and \cite{EMSW20}, who considered meagerness in a similar topological space to study non-definability of rings of integers in infinite extensions of~$\qq$.

In order to prove \Cref{TI:meager}, we first establish a new necessary criterion for the existence of universal quadratic forms in a given number of variables, which is of independent interest, in particular also for the study of universal quadratic forms over number fields.

For a commutative ring $R$, denote by $R^\times$ the set of units in $R$, and by $R^{\times 2}$ the set of squares in $R^\times$.
For an algebraic extension $K/\qq$, denote by $\mc{O}_K$ the ring of algebraic integers in $K$, and for $K \in \mc X$, 
denote by $\co_K^+$ the set of totally
positive elements in $\co_K$, and let $\co_K^{\times,+}= \co_K^\times\cap \co_K^+$ 
(recall that $\mc X$ denotes the set of all subfields of $\Q^{\tr}$, and so each $K\in\mc X$ is totally real).
We have $\mc{O}_K^{\times 2} \subset \mc{O}_K^{\times, +} \subset \mc{O}_K^\times \subset K^\times$, and all of them are groups with respect to multiplication.
We will say that a quadratic form over $K$ \emph{represents} a class of $\mc{O}_K^{\times, +}/\mc{O}_K^{\times 2}$ if it represents at least one element of $\mc{O}_K^{\times, +}$ in the given equivalence class. 
We then establish the following.

\begin{thmI}[see \Cref{thm:rscl}]\label{TI:units-universal}
Let $K \in \mc X$ and $n \geq 2$.
A positive definite $n$-ary quadratic form over $K$ represents at most $2n - 2$ classes of $\mc{O}_K^{\times, +}/\mc{O}_K^{\times 2}$.
\end{thmI}

In particular, if $\mc{O}_K^{\times, +}/\mc{O}_K^{\times 2}$ is infinite, then no universal quadratic form over $K$ can exist.

\Cref{TI:units-universal} is folklore for \textit{classical} positive definite quadratic forms (i.e. whose cross-terms are all divisible by 2) over a totally real number field $K$; 
recently, the first three authors \cite[Proposition 4.2]{DKM-Northcott} 
wrote down a proof that also explicitly covers infinite extensions. 

The significant novelty of our method is thus that it also covers arbitrary positive definite forms.
In particular, it implies a new result towards Kitaoka's conjecture: \textit{If a totally real number field has $|\mc{O}_K^{\times, +}/\mc{O}_K^{\times 2}|>4$ (i.e. it has a system of fundamental units, among which at least $3$ are totally positive), then it admits no ternary universal quadratic form.}
We do not know whether the bound $2n - 2$ in \Cref{TI:units-universal} is optimal for $n \geq 3$, see \Cref{ques:2n-2} below.

We will compare the property that a totally real field $K$ admits a universal quadratic form, to a property regarding the distribution of algebraic integers in $K$.
We consider the \textit{house} of an algebraic integer $\alpha$, defined as 
$\house{\alpha}= \max_i(|\alpha_i|)$, where 
 $\alpha_1,\dots,\alpha_n$ are all the conjugates of  $\alpha$. We then say that $\mc{O}_K$ has the \emph{Northcott property} (with respect to the house) if, for every $r \in \rr$, there exist only finitely many $\alpha \in \mc{O}_K$ with $\house{\alpha} < r$ (see Section \ref{sect:examples} for more details). 

Bounding the house  and the degree of an algebraic integer implies a bound on the modulus of the coefficients of its monic minimal polynomial in $\Z[x]$. Therefore, $\mc{O}_K$ has the Northcott property for any number field $K$. An example for a field of infinite degree was
given already in 1962 by Robinson \cite{Rob62} when she showed that the ring of integers of $\Q^{\tr,2}$
has the Northcott property, and used it to
conclude that the first order theory of this ring is undecidable.
The term {Northcott property} was coined in 2001 by Bombieri--Zannier \cite{Bombieri-Zannier} and since then, its study and applications have become a highly active research area, see e.g. \cite{Checcoli-Fehm,GR,U2,PTW,Spr,Vidaux-Videla,Widmer11,Widmer}. 

Northcott property and universal forms are not only significant from the viewpoint of arithmetic statistics or Diophantine geometry, but recent results also connected them
by showing 
that a totally real field $K$ of infinite degree over $\qq$ has no universal quadratic form  \cite{DKM-Northcott} (or a form of higher degree \cite{Prakash}) 
whenever $\mc{O}_K$ has the Northcott property with respect to the house.
We answer in the negative the open question  whether the converse holds.

\begin{thmI}[see \Cref{prp:Northcott-meager}]\label{TI:northcott}
    The set of totally real fields for which the ring of integers has the Northcott property with respect to the house is a meager subset of $\mc{X}$.
\end{thmI}

As the union of two meager sets is again meager, it follows that for most  $K \in \mc X$, there is no universal form and $\mc{O}_K$ does not have the Northcott property.
We shall construct a concrete example of such a field in \Cref{prp-noNP-noUQF}.

This paper is structured as follows.
In \Cref{sect:tot-real-0-4} we isolate some computations on small totally positive integers in totally real fields.
In \Cref{sect:universal-qf-units} we give a proof of a slightly more general version of \Cref{TI:units-universal}: we will consider positive definite \textit{quadratic lattices}, which generalize quadratic forms.
Even when one cares only about quadratic forms, it will be crucial for the induction step in the proof of \Cref{TI:units-universal} that we allow the more general quadratic lattices.
In \Cref{sect:examples} we apply this result to obtain some new examples of totally real fields without a universal quadratic form (or lattice), and also collect further information on the Northcott property. 
In \Cref{sect:topological} we discuss the topology on $\mc X$ in more detail before proceeding to prove Theorems \ref{TI:meager} and \ref{TI:northcott}.
We conclude by briefly addressing in \Cref{sect:non-tot-real} the existence of universal forms over non-totally real algebraic extensions of $\qq$.

\section*{Acknowledgments}

We thank Martin Čech, Stevan Gajović, Jakub Krásenský, and Dayoon Park for interesting and helpful discussions about the work,
and the anonymous referee for several very helpful comments and suggestions.

\section{Totally real algebraic integers between 0 and 4}\label{sect:tot-real-0-4}
A complex number $\alpha \in \cc$ which is algebraic over $\qq$ will be called an \textit{algebraic number}, and if it is integral over $\zz$ it will be called an \textit{algebraic integer}.
We will study algebraic field extensions $K/\qq$, which we consider as subfields of $\cc$, and will denote by $\mc{O}_K$ the ring of integers of $K$, i.e.~the set of algebraic integers in $K$.

We call an algebraic number $\alpha$ \textit{totally real} if all of its conjugates over $\qq$ are real.
We call an algebraic field extension $K/\qq$ \textit{totally real} if all of its elements are totally real.
We denote the field of all totally real numbers by $\qq^{\tr}$, and we write $\mc{O}^{\tr}$ for the ring of integers of $\qq^{\tr}$.

Given a totally real field $K$ and $\alpha, \beta \in K$, we write $\alpha \succ \beta$ (or $\beta \prec \alpha$) to say that $\sigma(\alpha) > \sigma(\beta)$ for every embedding $\sigma : K \hookrightarrow \rr$.
This does not depend on the choice of the field $K$, as every embedding $K \hookrightarrow \rr$ extends to an embedding $\qq^{\tr} \hookrightarrow \rr$.
We call a totally real algebraic number $\alpha$ \textit{totally positive} if $\alpha\succ 0$.
As in the Introduction, we denote by $\mc{O}_K^\times$, $\mc{O}_K^{\times 2}$ and $\mc{O}_K^{\times, +}$ the sets of \textit{units}, \textit{squares of units}, and \textit{totally positive units} in $\mc{O}_K$, respectively.

$\mc{X}$ denotes the set of all totally real fields, viewed as subfields of $\qq^{\tr}$. (In \Cref{sect:topological}, we will define a topology on $\mc X$, but we do not need it until then.)

We will use the fact that totally real algebraic integers $\alpha$ with $0 \prec \alpha \prec 4$ can be nicely characterized.

\begin{lem}\label{cor:smtpe} 
If $\alpha$ is a totally real algebraic integer with $0 \prec \alpha \prec 4$, then there exists a root of unity $\zeta$ such that $\alpha = (\zeta + \zeta^{-1})^{2} = 2 + \zeta^2 + \zeta^{-2}$.
\end{lem}

\begin{proof}
This is quite well-known. First, we know \cite[Lemma 3]{McK08CA}  that: \textit{
The map $\theta: z \mapsto z+z^{-1}$ establishes a bijection between algebraic integers all of whose conjugates lie on the unit circle 
and totally real algebraic integers all of whose conjugates lie  in the real interval $[-2,2]$.}
Thanks to a classical theorem of Kronecker, the former of these are just the roots of unity. Thus we get:
\textit{If $\beta$ is a totally real algebraic integer with $-2 \prec \beta\prec 2$, then $\beta = \zeta+\zeta^{-1}$ for some root of unity $\zeta$.}
The lemma follows by applying this last result to $\sqrt\alpha$.
\end{proof}

In order to work with the elements from \Cref{cor:smtpe} more conveniently, we write for $a \in \Q$
\[
g(a) = e^{2\pi ia} + e^{-2\pi ia} = 2\cos(2\pi a) \in \mc{O}^{\tr}.
\]
We gather some computation rules, all of which are easily verified from the definition and which we will mostly use without explicit reference in the rest of the section.

\begin{lem}\label{lem:comp-rules}
Let $a, b \in \Q$.
We have
\setlength{\columnsep}{-1cm}
\begin{multicols}{2}
\ber
    \item\label{it:mf} $g(a)g(b) = g(a+b)+g(a-b)$, 
    \item\label{it:1/2} $g(a) = g(-a) = g(a+1)$,
    \item $g(a) = -g(a+\frac{1}{2})$,
    \item\label{it:sq} $g(a)^2 = 2 + g(2a)$,
    \item\label{it:4-} $4 - g(a)^2 = g(a+\frac{1}{4})^2$,
    \item $g(a) = 0$ if and only if $a = \frac{s}{4}$ for $s \in \Z$ odd.
\ee
\end{multicols}
\end{lem}

\begin{lem}\label{lem:even-odd-multp}
Let $K \in \mc X$, $a \in \Q$ be such that $g(a)^2 \in K$.
Then for $c \in \Z$ we have
\begin{equation*}
\begin{cases}
    g(c a) \in K &\text{if } c \text{ is even}, \\
    g(a)g(c a) \in K &\text{if } c \text{ is odd}.
\end{cases}
\end{equation*}
\end{lem}

\begin{proof}
Since $g(ca) = g(-ca)$, it suffices to show this for $c \geq 0$.
The cases $c=0$ and $c=1$ are trivial, since $g(0) = 2 \in K$ and $g(a)^2 \in K$.
Now assume $c \geq 2$, we continue by induction.

By the multiplication formula \ref{it:mf} from \Cref{lem:comp-rules} we have
$$g(ca) = g(a)g((c-1)a) - g((c-2)a).$$
If $c$ is even, then by induction hypothesis $g(a)g((c-1)a), g((c-2)a) \in K$ and thus $g(ca) \in K$.
If $c$ is odd, then we multiply both sides by $g(a)$ to obtain
$$ g(a)g(ca) = g(a)^2 g((c-1)a) - g(a)g((c-2)a)$$
and use that, by assumption and induction hypothesis, $g(a)^2, g((c-1)a)$ and $g(a)g((c-2)a) \in K$.
\end{proof}

\begin{lem}\label{lem:cyclotomic-2power}
Let $K \in \mc X$, let $m \geq 0$ and consider $s, t \in \Z$ odd such that $g(\frac{s}{2^mt})^2 \in K$.
Then either $g(\frac{s}{2^mt}) \in K$ or $g(\frac{1}{2^{m}})g(\frac{s}{2^mt}) \in K$.
\end{lem}

\begin{proof}
Let $s, t \in \Z$ be odd and coprime and such that $g(\frac{s}{2^mt}) \not\in K$ but $g(\frac{s}{2^mt})^2 = 2 + g(\frac{2s}{2^mt}) \in K$.
By \Cref{lem:even-odd-multp} (with $a = \frac{s}{2^mt}$ and $c \in \Z$ such that $sc \equiv t \bmod 2^mt$), we obtain $g(\frac{1}{2^m})g(\frac{s}{2^mt}) \in K$ as desired.
\end{proof}

Given $K \in \mc X$, by \ref{it:sq} of \Cref{lem:comp-rules}, we either have $g(\frac{1}{2^m}) \in K$ for all $m \geq 0$, or there exists a unique $m > 0$ such that $g(\frac{1}{2^m})^2 \in K$ but $g(\frac{1}{2^m}) \not\in K$.
Hence, the following definition makes sense.

\begin{dfn}\label{dfn:gamma_K}
Let $K \in \mc X$.
We denote $\gamma_K = g(\frac{1}{2^m})^2$ if $m > 0$ is such that $g(\frac{1}{2^m})^2 \in K$ but $g(\frac{1}{2^m}) \not\in K$, or $\gamma_K = -1$ if no such $m$ exists.
\end{dfn}

\begin{prp}\label{prp:ingredients-cyclotomics}
Let $K \in \mc X$.
Let $\gamma_K$ be as in \Cref{dfn:gamma_K}.
Consider $\alpha \in \mc{O}_K$ with $0 \prec \alpha \prec 4$ and set $\beta = 4 - \alpha$.
Assume that $\alpha \not\in K^{\times 2}$.
Then we have:
\ber    
    \item If $\beta \in K^{\times 2}$, then $\beta \in \mc{O}_K^{\times 2}$.
    \item If $\beta \not\in K^{\times 2}$, then $\alpha \gamma_K \in K^{\times 2}$ and $\beta \gamma_K \in K^{\times 2}$.
\ee
\end{prp}

\begin{proof}
In view of \Cref{cor:smtpe}, write $\alpha = g(a)^2$ for some $a \in \Q$.
Write $a = \frac{s}{2^m t}$ for odd, coprime $s, t \in \Z$ and $m \geq 0$.
Since $g(1) = -g(\frac{1}{2}) = 2 \in K^\times$, we obtain by \Cref{lem:cyclotomic-2power} that $m \geq 2$.

We now make a case distinction depending on whether or not $m = 2$.
Recall that $\beta = g(a+\frac{1}{4})^2$.
If $m = 2$, then $a + \frac{1}{4}$ has a representation in reduced form with denominator not divisible by $4$; as in the previous paragraph we obtain that $g(a+\frac{1}{4}) \in K$ and thus $\beta \in K^{\times 2}$.
Since $g(\frac{1}{4}) = g(\frac{3}{4}) = 0$ and $g(0) = -g(\frac{1}{2}) = 2$, but our hypotheses imply both $\beta$ and $4 - \beta$ are non-zero, there must be an odd prime number dividing the denominator of $a+\frac{1}{4}$.
Now we use a well-known result (see e.g. \cite[Section 4]{R62}) saying that: \textit{Let $b = u/v \in \Q$ with $(u,v)=1$. Then $g(b)$ is a unit, unless $v=2^k$, $k\ge 0$, or $v = 4p^k$ where $p$ is an odd prime and $k\in\Z_{\geq 1}$.}
We thus have $\beta= g(a+\frac{1}{4})^2 \in \mc{O}_K^{\times 2}$.

Now assume $m > 2$, hence in particular $g(\frac{1}{2^m}) \neq 0$.
By \Cref{lem:cyclotomic-2power} we obtain that $\gamma_K = g(\frac{1}{2^m})^2$ and $\alpha \gamma_K \in K^{\times 2}$.
By switching the roles of $\alpha$ and $\beta$, we obtain that also $\beta \gamma_K \in K^{\times 2}$.
\end{proof}

\section{Universal quadratic lattices and totally positive units}\label{sect:universal-qf-units}
Let $K/\qq$ be an algebraic field extension.
A \emph{quadratic lattice over $K$} is a pair $(\mc L, Q)$ where $\mc L$ is a torsion-free finitely generated $\mc{O}_K$-module, and $Q$ is a map $\mc L \to \mc{O}_K$ such that
\begin{itemize}
\item $Q(av) = a^2Q(v)$ for $a \in \mc{O}_K$ and $v \in \mc L$, and
\item the associated map
$$ B_Q : \mc L \times \mc L \to \frac{1}{2}\mc{O}_K : (v, w) \mapsto \frac 12\left(Q(v+w) - Q(v) - Q(w)\right)$$
is bilinear.
\end{itemize}
Given such a quadratic lattice, the $K$-dimension of $\mc L \otimes_{\mc{O}_K} K$ is called the \emph{rank of $(\mc L, Q)$}.

For example, for a natural number $n$, if $f(X_1, \ldots, X_n) \in \mc{O}_K[X_1, \ldots, X_n]$ is a homogeneous quadratic polynomial with coefficients in $\mc{O}_K$, i.e.~a quadratic form over $K$, then it naturally makes the free $\mc{O}_K$-module $\mc{O}_K^n$ into a quadratic lattice (of rank $n$) via the map $Q_f : \mc{O}_K^n \to \mc{O}_K : (x_1, \ldots, x_n) \mapsto f(x_1, \ldots, x_n)$.
Up to isomorphism, every quadratic lattice structure on a free finitely generated $\mc{O}_K$-module can be obtained in such a way, but in general, not every torsion-free finitely generated $\mc{O}_K$-module is free.
Quadratic lattices hence form a proper generalization of quadratic forms over $K$.
\begin{rmk}\label{rem}
We note that slight variations in terminology occur in the literature, in particular, the notion we call ``quadratic lattice over $K$'' might sometimes be called a ``quadratic lattice over $\mc{O}_K$'' or ``quadratic $\mc{O}_K$-lattice'' to emphasize the base ring of the module $\mc L$, or also an ``integral quadratic lattice'' to emphasize that $Q$ takes only values in $\mc{O}_K$.

With these caveats in mind, the reader may consult any textbook on quadratic forms and lattices over rings of integers in number fields (like \cite[Part Four]{OMe00}) to learn more about them, or \cite[Section 2]{DKM-Northcott} for a presentation specifically including algebraic extensions $K/\qq$ of infinite degree.
\end{rmk}

Let $(\mc L, Q)$ be a quadratic lattice over $K$.
An element $d \in \mc{O}_K$ is \emph{represented by $(\mc L, Q)$} if $d = Q(v)$ for some $v \in \mc L$.
When $K$ is totally real, then we say that the lattice $(\mc L, Q)$ is
\begin{itemize}
\item \emph{positive definite} if $Q(v) \succ 0$ for all $v \in \mc L \setminus \lbrace 0 \rbrace$,
\item \emph{universal} if it is positive definite and represents every totally positive element of $\mc{O}_K$.
\end{itemize}
This generalizes the notions for quadratic forms mentioned in the Introduction.

When $(\mc{L}, Q)$ is a quadratic lattice over $K$ and $\mc{L}'$ is an $\mc{O}_K$-submodule of $\mc L$, then by restricting the domain of $Q$ to $\mc{L'}$ we obtain a quadratic lattice $(\mc{L'}, Q\vert_{\mc{L}'})$, which we call a \emph{sublattice of $(\mc{L}, Q)$}.

\begin{thm}\label{thm:rscl} 
Let $K$ be a totally real field. Then a positive definite quadratic lattice $(\mc L, Q)$  over $K$ of rank $n\ge 2$ represents at most  $2n-2$ classes of $\mc{O}_K^{\times, +}/\mc{O}_K^{\times 2}$.

Furthermore, if $\gamma_K K^{\times 2} \cap \mc{O}_K^{\times,+} = \emptyset$, then $(\mc L, Q)$ represents at most $n$ classes of $\mc{O}_K^{\times, +}/\mc{O}_K^{\times 2}$.
\end{thm}
\begin{proof}
We apply induction on the rank $n$ of the lattice $(\mc L, Q)$.

Suppose $n\geq 2$. We may assume $(\mc L, Q)$ represents at least one unit; otherwise there is nothing to show. After rescaling $Q$ by this unit, we may assume $(\mc L, Q)$ represents $1$, that is, there is a vector $v\in \mc L$ such that $Q(v) = 1$. We consider the sublattice
\[
\mc L' = \cbm{ w\in \mc L}{B_Q(v,w) = 0} \subset \mc L,
\]
where $B_Q$ denotes the bilinear form associated to $Q$. 
Clearly, the rank of $\mc L'$ is at most $n-1$.

Suppose that $(\mc L,Q)$ represents a unit $\varepsilon$ which is not a square and not represented by $(\mc L', Q|_{\mc L'})$. First we make the observation that
\[
2\mc L \subset \mc O_K v + \mc L'.
\]
This is straightforward to verify: for any $w\in \mc L$ we have 
\[
2w = 2B_Q(v,w) v + \underbrace{(2w - 2B_Q(v,w)v)}_{\in \mc L'};
\]
note that $2B_Q(v,w)\in\mc{O}_K$.

Now write $\varepsilon = Q(x)$, $x\in \mc L$, and write $2x = av+w$ for some $a\in \mc O_K$ and $w\in \mc L'$. We compute
\[
4\varepsilon = Q(2x) = Q(av+w) = Q(av) + Q(w) = a^2 + Q(w).
\]
Since $Q$ is positive definite, this implies
\[
0 \preceq \varepsilon^{-1} a^2 \preceq 4.
\]

We first show that both inequalities above have to be strict. 
\bi
\item If $\varepsilon^{-1}a^2=0$, then $2x = w \in \mc L'$, which implies $x\in \mc L'$. This says $\varepsilon$ is already represented by $(\mc L', Q|_{\mc L'})$, a contradiction. 
\item If $\varepsilon^{-1}a^2=4$, then $\varepsilon = (a/2)^2$ is a square, also a contradiction. 
\ei
So we have that $0 \prec \varepsilon^{-1} a^2 \prec 4$.

Since $\varepsilon$ is not a square, $\varepsilon^{-1}a^2$ is also not a square.
We apply \Cref{prp:ingredients-cyclotomics} (with $\alpha=\varepsilon^{-1}a^2$ and $\beta=\varepsilon^{-1}Q(w)$), and so we need to consider two cases:
\bi
\item $\varepsilon^{-1}Q(w) = \mu^2$ for some $\mu \in \mc{O}_K^\times$; in this case $\varepsilon = Q\rb{\mu^{-1} w}$ is represented by $(\mc L', Q|_{\mc L'})$, a contradiction.
\item $\varepsilon^{-1}a^2 \gamma_K \in K^{\times 2}$ and $\varepsilon^{-1}Q(w)\gamma_K \in K^{\times 2}$, and thus both $\varepsilon \gamma_K \in K^{\times 2}$ and $Q(w) \in K^{\times 2}$.
\ei
We conclude that the only non-square totally positive units represented by $(\mc L, Q)$ but not by $(\mc L', Q\vert_{\mc L'})$ must lie in $\gamma_K K^{\times 2}$; in particular, all such units lie in the same square class, and if $\gamma_K K^{\times 2} \cap \mc{O}_K^{\times, +} = \emptyset$, then such units simply cannot exist.

To summarize, we showed that there are at most 2 square classes of totally positive units represented by $(\mc L, Q)$ but not by $(\mc L', Q\vert_{\mc L'})$, namely, $1$ and the possible unit in $\gamma_K K^{\times 2}$.

For $n=2$, we may further refine the arguments to establish that $(\mc L,Q)$ represents at most 2 square classes of totally positive units. In this case, $\mc L'$ is a lattice of rank $1$, so $\mc L' \otimes_{\mc{O}_K} K$ is generated by a single vector $z$, and every element of $\mc L'$ is thus of the form $az$ for some $a \in K$.
Since $Q(az) = a^2Q(z)$, $(\mc L',Q|_{\mc L'})$ represents at most 1 square class of units.

Now if $(\mc L,Q)$ does \textit{not} represent a unit $\varepsilon$ which is not a square and not represented by $(\mc L', Q|_{\mc L'})$, then it represents at most 2 square classes of units. 
If there exists such an $\varepsilon$, then we can proceed exactly as above to conclude that $Q(w)$ is a square. As $(\mc L', Q\vert_{\mc L'})$ is of rank $1$ and represents the square $Q(w)$,
the only square class of totally positive units that can be represented by $(\mc L', Q|_{\mc L'})$ is $1$. We again conclude that $(\mc L,Q)$ represents at most 2 square classes of units.

Starting from the case $n=2$, invoking the induction hypothesis for the sublattice $(\mc L', Q|_{\mc L'})$ establishes the theorem.
\end{proof}
\begin{eg}
The extra condition $\gamma_K K^{\times 2} \cap \mc{O}_K^{\times,+} = \emptyset$ in \Cref{thm:rscl} is satisfied, for example, when the ramification of $2$ in $K$ is odd.
Indeed, in this case $\gamma_K = g(\frac{1}{8})^2 = 2$ (since $g(\frac{1}{16})^2 = 2 + \sqrt{2} \not\in K$), and $2K^{\times 2}$ cannot contain any units.
In particular, this applies whenever $[K : \mathbb{Q}]$ is odd, or $2$ is unramified in $K$.
\end{eg}

We in fact do not know of any examples of a  field $K\in\mc{X}$ and a rank $n$ positive definite quadratic lattice over $K$ representing strictly more than $n$ classes in $\mc{O}_K^{\times, +}/\mc{O}_K^{\times 2}$.
Thus, we cannot rule out that the bound $2n-2$ from \Cref{thm:rscl} could be sharpened to $n$ in general.
\begin{ques}\label{ques:2n-2}
Does there exist a totally real field $K$ and a positive definite quadratic lattice  $(\mc L, Q)$  over $K$ of rank $n$ which represents at least $n+1$ classes in $\mc{O}_K^{\times, +}/\mc{O}_K^{\times 2}$?
\end{ques}

\section{Some examples}\label{sect:examples}

For $n\ge 2$, let $\Q^{\tr,n}$ denote the compositum of all totally real extensions of $\Q$ of degree $\le n$.

\begin{prp}\label{prp:Q23nuf} 
The fields $\Q^{\tr,2}$ and $\Q^{\tr,3}$ admit no universal quadratic forms.
\end{prp}
\begin{proof}
With \Cref{thm:rscl}, it suffices to show that $\Q^{\tr,2}, \Q^{\tr,3}$ have infinitely many square classes of totally positive units. For $\Q^{\tr,2}$, this is just \cite[Theorem 7.3]{DKM-Northcott}. Meanwhile, for $\Q^{\tr,3}$, we note that for any cubic field $L$, the extension $\Q^{\tr,2}L/\Q^{\tr,2}$ is always a cubic Galois extension. This says the field extension $\Q^{\tr,3}/\Q^{\tr,2}$ contains no subextensions of even degree over $\Q^{\tr,2}$. So the infinitely many square classes of totally positive units from $\Q^{\tr,2}$ are preserved in $\Q^{\tr,3}$ (cf.~\cite[Proposition 4.4]{DKM-Northcott}).
\end{proof}

\begin{rmk}
We expect $\Q^{\tr,n}$ to have infinitely many square classes of units for all $n\ge 2$, and hence admit no universal quadratic forms. But it seems quite difficult to explicitly construct these units.
\end{rmk}

In \cite[Theorem 1.1]{DKM-Northcott} a different proof was given for the fact that $\Q^{\tr,2}$ admits no universal quadratic form, relying on the Northcott property.
Let us recall what we mean by this: for an algebraic number $\alpha$, we denote by $\house{\alpha} \in \rr$ its \textit{house}, i.e.~the largest of the absolute values of all its conjugates.
For a set of algebraic numbers $A$, we say that $A$ has the \textit{Northcott property} (with respect to the house) if, for all $r \in \R$, the set $ \lbrace x \in A \mid \house{a} < r \rbrace$ is finite (or empty).

Note that various height functions are used through the literature, in particular the logarithmic Weil height. This does not make a difference for our purposes as, if $\co_K$ has the Northcott property with respect to the logarithmic Weil height, then it has it also with respect to the house (see e.g. \cite[Section 3]{DKM-Northcott} for more details). Further, while \cite{DKM-Northcott} was formulated in terms of the Northcott property of $\co_K^+$, 
this is equivalent to our formulation with the Northcott property of $\co_K$, as one quickly sees by
replacing elements $\alpha\in\co_K$ by $\lceil\house{\alpha}\rceil+\alpha\in\co_K^+$. 

One of the main results of \cite{DKM-Northcott} was the following.

\begin{thm}[{see \cite[Theorem 3.1]{DKM-Northcott}}]
Let $K/\qq$ be a totally algebraic extension of infinite degree.
If $\mc{O}_K$ has the Northcott property, then there is no universal quadratic lattice over $K$.
\end{thm}

Robinson \cite{Rob62} showed that $\mc{O}_{Q^{\tr,2}}$ has the Northcott property. 
However, for each $n\geq 3$ it is unknown whether $\mc{O}_{\Q^{\tr,n}}$ has the Northcott property (broader question on this was first raised in \cite{Bombieri-Zannier}).

We shall now construct the promised example of a totally real field $K$ for which $\mc{O}_K$ does not have the Northcott property, nor a universal quadratic form.
In the next section, we shall show that, in a certain sense, for most totally real fields $K$, $\mc{O}_K$ neither has the Northcott property nor admits a universal quadratic form.

\begin{lem}\label{lem-Dirichlet}
There exists an infinite increasing  sequence  of primes  $(q_i)_i$ such that $q_i\equiv 3\pmod 4$ and $(q_i-1,q_j-1)=2$ whenever $i\neq j$.
\end{lem}
\begin{proof}
Set $q_1=7$. Next, suppose $i\in \Z_{> 0}$ and $q_1,\ldots,q_i$ are given. Let $p_1,\ldots, p_k$ be all the odd prime divisors of $(q_1-1)\cdots (q_i-1)$.
Consider the system of congruences
\begin{alignat*}1
x&\equiv 3 \pmod 4\\
x&\equiv 2 \pmod {p_1}\\
&\vdots\\
x&\equiv 2 \pmod {p_k}.
\end{alignat*} 
By the Chinese Remainder Theorem there exists a solution $x_0$.  For each $n\in \Z_{> 0}$ the  element $x_n=x_0+n\cdot 4p_1\cdots p_k$ is also a solution.
Since $(x_0,4p_1\cdots p_k)=1$ it follows from Dirichlet's Theorem on primes in arithmetic progressions that there are infinitely many $n\in \Z_{> 0}$ such that
$x_{n}$ is prime. We pick one of those  satisfying $x_n>q_i$ and set $q_{i+1}=x_n$.
\end{proof}
\begin{rmk}\label{rm-Dirichlet}
Following the notation of \Cref{sect:tot-real-0-4}, for an odd prime $q$, the field $\qq(g(\frac{1}{q}))$ is a totally real field with $[\qq(g(\frac{1}{q})) : \qq] = \frac{q-1}{2}$ \cite[p. 305]{R62}.

It follows that if $(q_i)_i$ is a sequence of primes as in \Cref{lem-Dirichlet}, then for $i \neq j$, the degrees $[\qq(g(\frac{1}{q_i})) : \qq]$ and $[\qq(g(\frac{1}{q_j})) : \qq]$ are coprime.
\end{rmk}
\begin{prp}\label{prp-noNP-noUQF}
Let $(q_i)_i$ be a sequence of primes as in \Cref{lem-Dirichlet} and let $F_i = \qq(g(\frac{1}{q_i}))$.
Define $K_0 = \Q^{\tr,2}$ and inductively define $K_i$ to be the compositum $K_{i-1}F_i$.
Set $K = \bigcup_{i \in \Z_{\geq 0}} K_i$.
The following hold:
\begin{enumerate}
\item\label{it:tot-real} $K$ is a totally real field,
\item\label{it:no-uqf} $\mc{O}_K^{\times, +}/\mc{O}_K^{\times 2}$ is infinite, and in particular there is no universal quadratic lattice over $K$,
\item\label{it:no-NP}
$\mc{O}_K$ does not have the Northcott property.
\end{enumerate}
\end{prp}

\begin{proof}
\eqref{it:tot-real} Since $K_0$ and each $F_i$ are totally real, so is $K$.

\eqref{it:no-uqf} We first show that $[K_i : K_0] = \prod_{j=1}^i [F_i : \qq]$.
To this end, first note that there exist a number field $M\subset K_0$ such that  $[K_i:K_0]=[MF_{i}\cdots F_{1}:M]$.
Next note that $[F_i\cdots F_1:\Q]=[F_i:\Q]\cdots [F_1:\Q]$ which follows immediately from the fact that the 
factors on the right hand-side are pairwise coprime. In particular, $[F_i\cdots F_1:\Q]$ is odd and of course it divides
$[MF_i\cdots F_1:\Q]$. Since $[MF_i\cdots F_1:\Q]=[MF_i\cdots F_1:M][M:\Q]$ and $[M:\Q]$ is a power of $2$ we conclude
that $[F_i\cdots F_1:\Q]$ divides $[MF_i\cdots F_1:M]$ which implies equality. This proves the claim.

In particular, we have that $[K_i : K_0]$ is odd for all $i$, and thus $K/K_0$ is a direct limit of finite extensions of odd degree.
By \cite[Theorem 7.3]{DKM-Northcott} the set $\mc{O}_{K_0}^{\times, +}/\mc{O}_{K_0}^{\times 2}$ is infinite, and then by \cite[Proposition 4.4]{DKM-Northcott} also the set $\mc{O}_K^{\times, +}/\mc{O}_K^{\times 2}$ is infinite, which in view of \Cref{thm:rscl} implies that there is no universal quadratic lattice over $K$.

\eqref{it:no-NP} The elements $g(\frac{1}{q_i})\in\co_K$ are pairwise distinct (as their respective degrees over $\Q$ are $(q_i-1)/2$) and they satisfy $-2 \prec g(\frac{1}{q_i}) \prec 2$. Hence $\mc{O}_K$ does not have the Northcott property.
\end{proof}

\section{Topological considerations}\label{sect:topological}

Let us recall some concepts from topology, in particular Baire theory. We refer the reader to any textbook on general topology, like \cite[Chapter 20]{HandbookAnalysis}, for details.
For a topological space $X$ and a subset $A \subset X$, we call $A$:
\begin{itemize}
\item \textit{nowhere dense} if the closure of $A$ does not contain a non-empty open set,
\item \textit{meager} if $A$ is a countable union of nowhere dense sets,
\item \textit{comeager} if $X \setminus A$ is meager,
\end{itemize}

Recall that $\qq^{\tr}$ denotes the field of totally real algebraic numbers.
We endow the power set $2^{\qq^{\tr}}$ with the product topology, where each component is endowed with the discrete topology; this is sometimes called the \emph{constructible topology}.
As a topological space it is homeomorphic to the Cantor space $2^\omega$; in particular (e.g.~by Tychonoff's Theorem) it is a compact Hausdorff space.
It is the coarsest topology on $2^{\qq^{\tr}}$ in which all sets of the form $\lbrace A \in 2^{\qq^{\tr}} \mid a \in A \rbrace$ for $a \in \qq^{\tr}$ are both open and closed.
In particular, any Boolean combination of such sets is both open and closed.
A basis of open sets of $2^{\qq^{\tr}}$ is given by the sets
\begin{equation}\label{eq:basis-open}
H(A, B) = \lbrace S \subset \qq^{tr} \mid A \subset S \text{ and } B \cap S = \emptyset \rbrace
\end{equation} 
for arbitrary finite sets $A, B \subset \qq^{tr}$.
Note that $H(A, B)$ is both open and closed.

We now consider within $2^{\qq^{\tr}}$ the subspace $\mc{X}$ consisting of all \textit{subfields} of $\qq^{\tr}$.
The subspace topology on $\mc{X}\subset 2^{\qq^{\tr}}$ is the \textit{constructible topology} already defined in the Introduction, i.e.  the coarsest topology in which the sets
$ \lbrace K \in \mc X \mid a \in K \rbrace $
are both open and closed, for all $a \in \qq^{\tr}$.
Note that $\mc{X}$ is a closed subset of $2^{\qq^{\tr}}$; this can be seen by observing that
\begin{align*}
\mc{X} = \enspace&\lbrace K \in 2^{\qq^{\tr}} \mid 0, 1 \in K \rbrace \\
&\cap \bigcap_{\substack{a \in \qq^{\tr} \\ b \in \qq^{\tr} \setminus \lbrace 0 \rbrace}} \lbrace K \in 2^{\qq^{\tr}} \mid \text{if } a, b \in K \text{ then } ab, a+b, -b, b^{-1} \in K \rbrace.
\end{align*}
In particular $\mc{X}$ with the subspace topology is also a compact Hausdorff space.
By the Baire category theorem, $\mc{X}$ is a Baire space, i.e.~countable unions of closed sets with empty interior also have empty interior.
Equivalently, countable intersections of dense open sets remain dense.
In particular there exist non-meager subsets of $\mc{X}$, and thus no subset of $\mc{X}$ is both meager and comeager.

We will now show that a totally real field $K \in \mc X$ ``usually does not carry a universal quadratic lattice'' by proving that the set of all $K \in \mc X$ for which a universal quadratic lattice exists, is meager inside $\mc X$.
This is inspired by the philosophy underlying \cite{EMSW20, DF21}, where it is shown that the set of algebraic field extensions $K/\qq$ for which $\mc{O}_K$ is first-order definable in $K$, is meager with respect to the constructible topology.

We shall use the following lemma, which is easily verified by using the open basis for the topology as given in \eqref{eq:basis-open}.

\begin{lem}\label{lem:dense}
A set $A \subset \mc X$ is dense in $\mc X$ if and only if, for every finite Galois extension of totally real number fields $L_0/K_0$ there exists $K \in A$ with $K \cap L_0 = K_0$.
\end{lem}
\begin{thm}\label{prp:uqf-meager}
Consider the subset of $\mc X$ given by
$$ U = \lbrace K \in \mc X \mid \mc{O}_K^{\times, +}/\mc{O}_K^{\times 2} \text{ is finite } \rbrace.$$
Then $U$ is a meager subset of $\mc X$.
In particular, its subset
$$ U' = \lbrace K \in \mc X \mid \text{ there is a universal quadratic lattice over } K \rbrace $$
is meager.
\end{thm}
\begin{proof}
Observe that, for $K \in \mc X$, one has
$$ K \in U \enspace\Leftrightarrow\enspace \exists S \subset \mc{O}_K^{\times, +} \text{ finite } \forall a \in \mc{O}_K^{\times, +} \exists c \in S : \sqrt{ac} \in K.$$
Denote by $\mc S$ the set of all finite subsets of $(\mc{O}^{\tr})^{\times, +}$ and note that this set is countably infinite.
By the above observation, we have
$$ U = \bigcup_{S \in \mc{S}} \bigcap_{a \in (\mc{O}^{\tr})^{\times,+}} \lbrace K \in \mc X \mid S \subset K, \text{ and if } a \in K, \text{ then } \exists s \in S : \sqrt{as} \in K \rbrace.$$
To show that $U$ is meager, since $\mc S$ is countable, it remains to show that $U_S = \bigcap_{a \in (\mc{O}^{\tr})^{\times,+}} \lbrace K \in \mc X \mid S \subset K, \text{ and if } a \in K, \text{ then } \exists s \in S : \sqrt{as} \in K \rbrace$ is nowhere dense, for each $S \in \mc S$.
And since $U_S$ is already closed, it remains to show that it has an empty interior, or equivalently, that its complement in $\mc X$ is dense.

To this end, we invoke \Cref{lem:dense}.
Consider $S \in \mc S$, an arbitrary number field $K_0$, and a finite Galois extension $L_0/K_0$.
We may assume without loss of generality that $S \subset K_0$, otherwise $K_0 \not\in U_S$ and we are done.
We may also extend $L_0$ and assume that all elements of $S$ are squares in $L_0$.
There are infinitely many natural numbers $n \equiv 1 \bmod 12$ such that $n(n+1)$, $3n(3n+4)$ and $(3n+3)(3n+4)$ are square-free integers: by applying the main theorem of \cite{Erd} to the polynomial $g(X) = (1+12X)(1+6X)(7+36X)$ one obtains that there exist infinitely many natural numbers $m$ for which $g(m)$ is square-free, and then for each such $m$, the natural number $n = 1+12m$ is as desired.
In particular, since $L_0/\qq$ is finite and thus contains only finitely many subextensions, we may find a natural number $n \equiv 1 \bmod 12$ such that $n(n+1)$, $3n(3n+4)$ and $(3n+3)(3n+4)$ are all square-free integers and such that $M = \qq(\sqrt{n(n+1)}, \sqrt{3n(3n+4)})$ is linearly disjoint from $L_0$, i.e.~$M \cap L_0 = \qq$.
By \cite[Proposition 6.6]{DKM-Northcott} there exists $\mu \in \mc{O}_M^{\times, +} \setminus M^{\times 2}\qq^\times$.
Now setting $K = MK_0 = K_0(\sqrt{n(n+1)}, \sqrt{3n(3n+4)})$, we have by \cite[Proposition 7.1]{DKM-Northcott} that $\mu$ is not a square in $ML_0$, and thus that for all $s \in S$, $\mu s$ is not a square in $K$.
We conclude that $K \not\in U_S$.
In view of \Cref{lem:dense}, we infer that $U_S$ has empty interior, as desired.
This establishes the meagerness of $U$ in $\mc X$.

For the second claim, it suffices to show that $U' \subset U$.
This is a direct consequence of \Cref{thm:rscl}: if some $K \in \mc X$ carries a universal quadratic lattice (i.e.~$K \in U'$), then this lattice has a certain finite rank $n$, but then $\mc O_K^{\times, +} / \mc{O}_K^{\times 2}$ contains at most $2n - 2$ elements, so $K \in U$.
\end{proof}

\begin{thm}\label{prp:Northcott-meager}
Consider the subset of $\mc{X}$ given by
\begin{displaymath}
N = \lbrace K \in \mc{X} \mid \forall n \in \Z_{>0} : \lvert \lbrace x \in \mc{O}_K \mid \house{x} < n \rbrace \rvert < \infty \rbrace,
\end{displaymath}
i.e.~totally real fields for which the ring of integers has the Northcott property (with respect to the house).
Then $N$ is a meager subset of $\mc{X}$.
In fact, already the superset
\begin{displaymath}
N' = \lbrace K \in \mc{X} \mid \lvert \lbrace x \in \mc{O}_K \mid \house{x} < 2 \rbrace \rvert < \infty \rbrace
\end{displaymath}
is meager.
\end{thm}

\begin{proof}
It suffices to show the second statement, since subsets of meager sets are meager.

Recall from \Cref{sect:tot-real-0-4} that $\lbrace x \in \mc{O}^{\tr} \mid \house{x} < 2 \rbrace$ is infinite; it consists of all numbers of the form $\zeta + \zeta^{-1}$ where $\zeta$ is a root of unity different from $\pm 1$.
Denote by $\mc{S}$ the set of all cofinite subsets of $\lbrace x \in \mc{O}^{\tr} \mid \house{x} < 2 \rbrace$, and note that $\mc{S}$ is countably infinite.
We observe that
$$ N' = \bigcup_{S \in \mc{S}} \bigcap_{s \in S} \lbrace K \in \mc X \mid s \not\in K \rbrace.$$
To show that $N'$ is meager, since $\mc{S}$ is countable, it remains to show that $N_S = \bigcap_{s \in S} \lbrace K \in \mc X \mid s \not\in K \rbrace$ is nowhere dense for every $S \in \mc{S}$.
Since $N_S$ is already closed, it remains to show that it has empty interior, or equivalently, that its complement in $\mc X$ is dense.
We want to invoke \Cref{lem:dense} to this end and show that its hypotheses are satisfied.

So fix $S \in \mc S$, and consider an arbitrary number field $K_0$ and a finite Galois extension $L_0/K_0$.
By \Cref{lem-Dirichlet} there is an infinite sequence of primes $(q_i)_i$ with $q_i \equiv 3 \pmod 4$ for all $i$, and $(q_i - 1, q_j - 1) = 2$ for $i \neq j$.
In particular, we may find a prime $q \equiv 3 \bmod 4$ such that $\frac{q-1}{2}$ and $[L_0 : \qq]$ are coprime and $g(\frac{1}{q}) \in S$.
By \Cref{rm-Dirichlet} we have $[\qq(g(\frac{1}{q})) : \qq] = \frac{q-1}{2}$, so in particular $[K_0(g(\frac{1}{q})) : K_0] = \frac{q-1}{2}$ and $[L_0 : K_0]$ are coprime.
But then $K_0(g(\frac{1}{q})) \cap L_0 = K_0$ and $K_0(g(\frac{1}{q})) \in \mc X \setminus N_S$.
In view of \Cref{lem:dense} we conclude that $\mc X \setminus N_S$ is dense in $\mc X$, as desired.
\end{proof}

\section{Non-totally real fields}\label{sect:non-tot-real}
So far, we have only considered algebraic extensions $K/\qq$ which are totally real; otherwise
the behaviour of universal quadratic lattices is radically different.
In fact, for a non-totally real algebraic field extension $K/\qq$, there \textit{always} exists a universal quadratic form, and one may even take the same rank $4$ positive definite form defined over $\qq$ for each field.
That such a form exists is possibly well-known to experts, but we include a short argument for the reader's convenience.

Note that here we call a quadratic lattice $(\mc L, Q)$ over an algebraic extension $K/\qq$ \textit{universal} if it represents precisely those elements $\alpha \in \mc{O}_K$ for which $\sigma(\alpha) \geq 0$ for all embeddings $\sigma:K \hookrightarrow \rr$.
If $K$ is totally real, this does not precisely coincide with the definition given in \Cref{sect:universal-qf-units}, as we now also allow positive \textit{semidefinite} lattices (i.e.~lattices which non-trivially represent $0$).
Nevertheless, in the example below, the lattice will not non-trivially represent $0$ as soon as there is at least one embedding $K \hookrightarrow \rr$.

\begin{prp}\label{prp:non-real}
Let $K/\qq$ be an algebraic field extension which is \emph{not} totally real.
The quadratic form over $K$ given by $W^2 + WX + X^2 + Y^2 + Z^2$ is universal.
\end{prp}
\begin{proof}
Let $Q(W, X, Y, Z) = W^2 + WX + X^2 + Y^2 + Z^2$.
Since $Q = (W + \frac{1}{2}X)^2 + \frac{3}{4}X^2 + Y^2 + Z^2$, it is clearly positive definite over $\qq$, and so it is positive semidefinite over $K$.

Given $\alpha\in\co_K$ such that $\sigma(\alpha) \geq 0$ for all embeddings $\sigma:K \hookrightarrow \rr$, we want to show that it is represented by $Q$. As $K$ is not totally real (but possibly has infinite degree), we can take a number field $F\subset K$ that is also not totally real and that contains $\alpha$. 

Since $Q$ has $\zz$-coefficients, we can naturally view it as a quadratic form over every completion $\mc O_\mathfrak p$ of $\mc{O}_F$ with respect to a prime ideal $\mathfrak p$.
Similarly, we can consider $Q$ as a quadratic form over the archimedean completions $\C$ and (possibly) $\R$ that arise from embeddings $\sigma:F\hookrightarrow\C$. When $R$ denotes one of the rings $\mc{O}_F$, $\mc O_\mathfrak p$, $\C$, or $\R$, we say that \textit{$Q$ represents an element $\beta\in R$  over $R$} if $Q(v)=\beta$ for some $v\in R^4$. 
As $F$ has at least one complex embedding and $Q$ has rank $4$, 
by the theory of spinor genera we have the following local--global principle for integral representations by $Q$ (see e.g.~\cite[p.~148]{Hsia76}): 
\textit{The quadratic form $Q$ represents an element $\gamma \in \mc{O}_F$ over $\mc{O}_F$ if and only if $Q$ represents $\gamma$ 
over every completion $\mc{O}_\mathfrak p$ for a prime ideal $\mathfrak p$ and $Q$ represents $\sigma(\gamma)$ over $\C$ (or $\R$) for every complex (or real) embedding $\sigma$.}

The form $Q$ represents all elements of $\C$, as they are represented already by $W^2$. By our assumption on $\alpha$, we have $\sigma(\alpha)\in\R_{\geq 0}$ for every real embedding $\sigma$, and so $Q$ represents $\sigma(\alpha)$ over $\R$.
It thus remains to show 
that $Q$ represent $\alpha$ over $\mc{O}_\mathfrak p$ for each prime ideal $\mathfrak p$.

If $\mathfrak p$ is non-dyadic (i.e.~ $2\not\in\mathfrak p$; then $2 \in \mc{O}_\mathfrak p^\times$), then by Hensel's Lemma, there exist $\beta_1, \beta_2 \in \mc{O}_\mathfrak p$ such that $-1 = \beta_1^2 + \beta_2^2$. Then we have
$$ \alpha = \left( \frac{\alpha +1}{2}\right)^2 + \left( \frac{\beta_1 (\alpha - 1)}{2}\right)^2 + \left( \frac{\beta_2 (\alpha - 1)}{2}\right)^2,$$
whereby already the subform $X^2 + Y^2 + Z^2$ of $Q$ represents $\alpha$ over $\mc{O}_\mathfrak p$.

Finally, let $\mathfrak p$ be dyadic (i.e.~ $2\in\mathfrak p$). 
For an element $\gamma \in \mc{O}_\mf{p}^\times$, let $\beta_1 \in \mc{O}_\mf{p}^\times$ be such that $\gamma \equiv \beta_1^2 \bmod \mf{p}$.
The polynomial $\beta_1^2 + \beta_1 X + X^2 - \gamma$ has a simple root $\beta_1$ modulo $\mf{p}$, and hence by Hensel's Lemma it has a root $\beta_2$ in $\mc{O}_\mf{p}$, whereby $\gamma = \beta_1^2 + \beta_1 \beta_2 + \beta_2^2$.
In particular, if $\alpha \in \mc{O}_\mf{p}^\times$, then it is represented over $\mc{O}_\mathfrak p$ by the subform $W^2 + WX + X^2$ of $Q$.
If $\alpha \in \mc{O}_\mathfrak p \setminus \mc{O}_\mathfrak p^\times$, then  $\alpha - 1 \in \mc{O}_\mathfrak p^\times$, and so $\alpha = \alpha - 1 + 1^2$ is represented over $\mc{O}_\mathfrak p$ by the subform $W^2 + WX + X^2 + Y^2$.
\end{proof}

\printbibliography

\end{document}